%% file: matroidconvexity.tex
\documentclass[a4paper,10pt]{amsart}
\usepackage[USenglish]{babel}
\usepackage[utf8]{inputenc}
\pdfoutput=1 

\input{definition.tex}

\usepackage[backend=bibtex8,style=alphabetic,eprint=true,maxcitenames =4,maxbibnames = 4]{biblatex}
\renewbibmacro{in:}{}

\title{Tropical linear spaces and tropical convexity}
\author{Simon Hampe}

\newcommand{\tconv}{\textnormal{tconv}}
\newcommand{\tpn}[1]{\mathbb{R}^{#1}/\textbf{1}}

\newcommand{\cone}{\textnormal{cone}}
\newcommand{\size}{\textnormal{het}}
\newcommand{\type}{\textnormal{part}}
\newcommand{\rank}{\textnormal{rank}}
\newcommand{\rec}{\textnormal{rec}}
\newcommand{\Imin}{I_{\min}}
\newcommand{\Imax}{I_{\max}}
\newcommand{\tnorm}[1]{\lVert #1 \rVert_\trop}
\definecolor{light-gray}{gray}{0.8}

\renewcommand{\epsilon}{\varepsilon}

\tikzset{
  >=latex  
}

\bibliography{bibliography}

\begin{document}

\begin{abstract}
 In classical geometry, a linear space is a space that is closed under linear combinations. In tropical geometry, it has long been a consensus that tropical varieties defined by valuated matroids are the tropical analogue of linear spaces. It is not difficult to see that each such space is tropically convex, i.e.\ closed under tropical linear combinations. However, we will also show that the converse is true: Each tropical variety that is also tropically convex is supported on the complex of a valuated matroid. We also prove a tropical local-to-global principle: Any closed, connected, locally tropically convex set is tropically convex.
\end{abstract}

\maketitle

\section{Introduction}

It has long been a consensus what the tropical analogue of a linear space should be. Sturmfels showed in \cite{ssolving} that the tropicalization of a complex variety defined by linear equations depends only on a matroid $M$ associated to these equations. One can give this tropical variety, the \emph{matroidal fan} or \emph{Bergman fan} of $M$ in various purely combinatorial ways: E.g.\ through its circuits, its bases or its lattice of flats \cites{fsmatroid,akbergman}. One can do this for any matroid, though only realizable matroids yield tropical varieties that are tropicalizations of algebraic linear spaces. In the case of fields with a nontrivial valuation, the tropicalization of a linear space is defined by a \emph{valuated matroid} $(M,w)$. This notion was originally introduced by Dress and Wenzel \cite{DressWenzel}. It is given by a matroid $M$ and the additional data of a \emph{valuation} $w$ on its bases. Again, the tropical space can be defined for any such object and it was soon established that the associated tropical varieties should be called \emph{tropical linear spaces}.

This terminology is further justified by the fact that being a tropical linear space is equivalent to this space having \emph{degree one}: This means that it intersects the linear space of complementary dimension associated to the uniform matroid  in exactly one point (where intersection is to be understood as \emph{stable intersection}). A proof of this can be found in \cite{fchowpolytopes}, though the statement seems to have been known for longer (see for example \cite{imstropicalgeometry}).

Someone familiar with tropical arithmetic might expect a different definition. In the algebraic world, a linear space is simply a space that is closed under linear combinations. On the tropical side, addition and multiplication are replaced by $\oplus = \max$ and $\odot = +$. Using these one can define tropical vector addition and tropical scalar multiplication. One might then be tempted to define a tropical linear space as a space that is closed under tropical linear combinations. This property is well-known under the name of \emph{tropical convexity}. At first glance, this might seem to be a misnomer, but its justification quickly becomes clear when looking at the corresponding literature. It turns out that classical and tropical convexity are closely related. Develin and Sturmfels first introduced the concept into the tropical world \cite{dstropicalconvexity} and proved --- among other things --- that there is a tropical Farkas' Lemma. Gaubert and Katz prove in \cite{gkminimalhalfspaces} that \emph{tropical polytopes}, i.e.\ the convex hulls of finitely many points, can also be written as the intersection of finitely many tropical halfspaces. Develin and Yu showed that tropical polytopes are tropicalizations of actual polytopes \cite{dytropicalcellular}. Tropical convexity has connections to many fields, such as graph theory, optimization, resolutions of monomial ideals or subdivisions of polytopes (see for example \cites{abgjtropsimplex,agjmeanpayoff,bycellular,frstiefeltropical,jlweighted}).

It becomes readily apparent that simply demanding tropical convexity will not in general produce sets that are tropical linear spaces in the approved sense. However, when adding the prerequisite that the set be supported on a tropical variety, i.e.\ be a \emph{balanced} polyhedral complex, the statement becomes true. In fact, it was already  well-known that any tropical linear space (meaning a space associated to a valuated matroid) is a tropical variety supported on a tropically convex set. In this paper, we prove that the converse is also true:

\begin{theorem}\label{main_theorem}
 Let $X$ be a tropical variety in $\tpn{n}$. Then $\abs{X}$ is tropically convex, if and only if $\abs{X} = B(M,w)$ for some valuated matroid $(M,w)$ on $[n]$. In other words, $X$ is the projectivisation of a space closed under tropical linear combinations if and only if $X$ is supported on a tropical linear space. 
\end{theorem}

In Section \ref{section_prelim} we will review basic definitions and facts about tropical convexity, tropical varieties, valuated matroids and their associated varieties. We also include a proof of the fact that a tropical linear space is tropically convex. In Section \ref{section_tconv} we collect results about general tropically convex complexes (i.e.\ that do not require balancing). We show that tropical convexity passes to recession fans and is  locally preserved. We also prove that any tropically convex fan of dimension $d$ is contained in the $d$-skeleton of the normal fan of the permutohedron. Section \ref{section_proof} then contains the actual proof of Theorem \ref{main_theorem}. We prove the statement first for fans and trivially valuated matroids. The general result then follows (with a bit more work) from the fact that being a tropical linear space is also equivalent to having a recession fan that is a tropical linear space. In section \ref{section_local}, we prove a tropical local-to-global convexity theorem:

\begin{theorem}\label{conv_theorem}
  Let $X \subseteq \tpn{n}$ be a closed, connected set. If $X$ is locally tropically convex, then $X$ is tropically convex.
\end{theorem}

From this we deduce that being a tropical linear space is a local property.

\begin{acknowledgement}
 The author was partially supported by DFG grants 4797/1-2 and JO366/3-2\footnote{This grant is part of the DFG priority project SPP 1489 (\url{www.computeralgebra.de})} and EPSRC grant EP/I008071/1. I would like to thank Michael Joswig for many helpful suggestions.
\end{acknowledgement}

\section{Preliminaries}\label{section_prelim}

\begin{convention}
 Throughout this paper we use $\oplus = \max$ as tropical addition. Of course, all results still hold in the $\min$-world, one simply has to \enquote{invert} the definition of tropical linear spaces as well. We also write $\odot = +$ for tropical multiplication.
\end{convention}

\subsection{Tropical convexity}

\begin{defn}
Let $x,y \in \R^n$. We define the tropical sum of $x$ and $y$ to be the componentwise tropical sum: $x\oplus y := (x_1 \oplus y_1,\dots,x_n \oplus y_n)$. Similarly, for $\lambda \in \R, x \in \R^n$ we define tropical scalar multiplication to be componentwise tropical multiplication: $\lambda \odot x := (x_1 + \lambda, \dots,x_n + \lambda)$.

 A subset $S$ of $\R^n$ is called \emph{tropically convex}, if for all $x,y \in S, \lambda,\mu \in \R$ we have
$$(\lambda \odot x) \oplus (\mu \odot y) \in S.$$

The \emph{tropical convex hull} of a set $T$, denoted by $\tconv(T)$, is the smallest tropically convex set containing $T$. By \cite{dstropicalconvexity} it is equal to the set of all tropical linear combinations of elements in $T$.
\end{defn}

\begin{remark}
 It is easy to see that any tropically convex set is invariant under translation by multiples of $(1,\dots,1)$. Hence it is customary to consider subsets $S'$ of the \emph{tropical projective torus} $\tpn{n}$, where $\textbf{1} := \gnrt{(1,\dots,1)}$. We say that such a set $S'$ is tropically convex, if its preimage under the quotient map $\R^n \to \tpn{n}$ is. Note, however, that tropical arithmetic operations are not actually well-defined on $\tpn{n}$.
\end{remark}

\begin{lemma}[{\cite[Proposition 3]{dstropicalconvexity}}]\label{lemma_tconv_pts}
 Let $x,y \in \tpn{n}$. Then the tropical convex hull $\tconv\{x,y\}$ is of the form $(\bigcup_{i=1}^k l_i)$, where the $l_i$ are consecutive line segments connecting $x$ and $y$, whose slopes are linearly independent $(0,1)$-vectors. Furthermore, the number of these line segments is $k := \abs{\{x_i - y_i; i=1,\dots,n\}}-1$.
\end{lemma}

\begin{remark}\label{prelim_remark_troplinesegment}
 Let us make this statement more concrete: Every element $x \in \tpn{n}$ has a well-defined \emph{heterogeneity}: 
  $$\size(x) := \abs{\{x_i; i=1,\dots,n\}}.$$
  We can also define a \emph{partition} associated to $x$, $\type(x) = I_1 \cup \dots \cup I_{\size(x)}$ of $[n]$ ordering the entries of $x$ descendingly. More precisely:  
   \begin{itemize}
  \item For any $j = 1,\dots,\size(x)$ and $k,l \in I_j$, we have $x_k = x_l =: x(I_j)$.
  \item If $j < j'$, then $x(I_j) > x(I_{j'})$.
 \end{itemize}

  Now fix representatives $x,y$ of two elements in $\tpn{n}$ and set $\Delta := \Delta(x,y) = y-x$. Assume that $\type(\Delta) = (I_1,\dots,I_s)$. For any set $F \subseteq [n]$ we write 
  $$e_F := \sum_{i \in F} e_i.$$
  If we define $F_j := \bigcup_{i=1}^j I_i$ for $j = 1,\dots,s-1$, then the tropical convex hull $\tconv\{x,y\}$ consists of the line segments connecting points $x = p_1,\dots,p_s$, where for $j > 1$ we set
  $$p_j := p_{j-1} + \left( \Delta(I_{j-1}) - \Delta(I_j) \right) e_{F_{j-1}} = x + \sum_{i=2}^j \left( \Delta(I_{i-1}) - \Delta(I_i) \right) e_{F_{i-1}}.$$
  In particular $p_s = \Delta(I_s) \odot y  \equiv y$ in $\tpn{n}$ and the slope of the line segment $[p_j,p_{j+1}]$ is $e_{F_j}$.
\end{remark}

\begin{ex}
 Let $n = 3$. We choose a representative of each element in $\tpn{3}$ by setting the first coordinate to be 0. Choose $x = (0,-1,-1), y = (0,2,1)$. Then $\Delta = (0,3,2)$, so we have $\type(\Delta) = ( \{2\},\{3\},\{1\})$. In particular, we get
 \begin{align*}
  p_1 &= x = (0,-1,-1),\\
  p_2 &= x + (3-2)e_{\{2\}} = x + (0,1,0) = (0,0,-1),\\
  p_3 &= (0,0,-1) + (2-0)e_{\{2,3\}} = (0,0,-1) + 2\cdot(0,1,1) = (0,2,1) = y.
 \end{align*}
 Note that there is a nice geometric way to construct this: Draw a \emph{min}-tropical line at each vertex. This induces a subdivision of the plane and the tropical convex hull then consists of the bounded cells of this subdivision. This also works for the tropical convex hull of an arbitrary (finite) number of vertices and in arbitrary dimensions, see \cite[Theorem 15]{dstropicalconvexity}.
\begin{figure}[ht]
 \centering
 \begin{tikzpicture}
  \draw[dotted] (-1,-1) -- (-2,-2);
  \draw[dotted] (-1,-1) -- (-1,2);
  \draw[dotted] (0,-1) -- (2,-1);
  \draw[dotted] (2,1) -- (2,2);
  \draw[dotted] (2,1) -- (3,1);
  \draw[dotted] (0,-1) -- (-1,-2);
  \draw (-1,-1) -- (0,-1) -- (2,1);
  \fill[black] (-1,-1) circle (2pt) node[left]{$p_1 = x = (-1,-1)$};
  \fill[black] (0,-1) circle (2pt) node[below right]{$p_2 = (0,-1)$};
  \fill[black] (2,1) circle (2pt) node[above right]{$p_3 = y = (2,1)$};
  
 \end{tikzpicture}
  \caption{Constructing the tropical convex hull of two points.}
\end{figure}
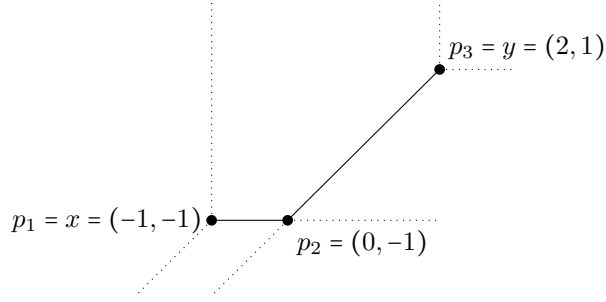

\end{ex}



\subsection{Tropical varieties}

We will only give the very basic definitions necessary for our purpose. For a general introduction to tropical geometry, see for example \cites{imstropicalgeometry,mrtropicalgeometry, MaclaganSturmfelsBook}.

\begin{defn}
 Let $X$ be a pure $d$-dimensional rational polyhedral complex in $\tpn{n}$. We will denote the \emph{support} of $X$ by $\abs{X} := \bigcup_{\sigma \in X} \sigma$.
 For a cell $\rho$ of $X$ we define $V_\rho := \gnrt{a-b; a,b \in \rho}$ to be the vector space associated to the affine space spanned by $\rho$ and we write $\Lambda_\rho := V_\rho \cap \Z^n/\textbf{1}$ for its lattice. 
 \begin{itemize}
  \item  Let $\sigma \in X^{(d)} := \{\sigma \in X; \dim(\sigma) = d\}$ and assume $\tau \subseteq \sigma$ is a face of dimension $d-1$. The \emph{primitive normal vector} of $\tau$ with respect to $\sigma$ is defined as follows: By definition there is a linear form $g$ such that its minimal locus on $\sigma$ is $\tau$. Then there is a unique generator of $\Lambda_\sigma/\Lambda_\tau \cong \Z$, denoted by $u_{\sigma/\tau}$, such that $g(u_{\sigma/\tau}) > 0$. 
 \item A \emph{tropical variety} $(X,\omega)$ is a pure, rational polyhedral complex $X$ together with a weight function $\omega: X^{\max} \to \N_{>0}$ (where $X^{\max}$ denotes the set of maximal polyhedral cells of $X$) fulfilling the \emph{balancing equation} at each codimension one cell $\tau$:
 $$\sum_{\sigma > \tau} \omega(\sigma) u_{\sigma/\tau} = 0 \textnormal{ mod } V_\tau.$$
 Note that we will consider two tropical varieties to be equivalent if they have a common refinement respecting the weight functions. A \emph{tropical fan} is a tropical variety whose polyhedral structure can be represented by a fan.
 \item For a polyhedral cell $\sigma$, we denote by 
 \begin{align*}
    \rec(\sigma) &:= \{v \in \tpn{n}: x + \R_{\geq 0}v \subseteq \sigma \textnormal{ for all } x \in \sigma\}\\
    &\,= \{v \in \tpn{n}: \exists x \in \sigma \textnormal{ such that } x + \R_{\geq 0}v \subseteq \sigma\}
 \end{align*}
its \emph{recession cone}. One can choose a refinement of a polyhedral complex $X$ such that $$\rec(X) := \{\rec(\sigma); \sigma \in X\}$$ is a fan and we will call that the \emph{recession fan} of $X$. If $(X,\omega)$ is a tropical variety, so is $(\rec(X),\omega_\rec)$, where
 $$\omega_\rec(\rho) = \sum_{\sigma: \rec(\sigma) = \rho} \omega(\sigma).$$
 (see \cite[p. 61]{rthesis} for a proof of this. It also follows implicitly from \cite[Theorem 5.4]{ahrrationalequivalence}.)
 \item A $d$-dimensional tropical variety $(X,\omega_X)$ is called \emph{irreducible}, if every $d$-dimensional variety $(Y,\omega_Y)$ with $\abs{Y} \subseteq \abs{X}$ is an integer multiple of $X$, i.e.\ $\abs{Y} = \abs{X}$ and (assuming we have chosen a common refinement) $\omega_Y = k \cdot \omega_X$ for some $k \in \N$. We also write this as $Y = k \cdot X$.
 \item Let $X$ be a tropical variety and $p \in \abs{X}$. By refining, we can assume without loss of generality that $p$ is a vertex of $X$. We define the \emph{Star} of $X$ at $p$ to be the fan 
 $$\Star_X(p) := \{ \R_{\geq 0} \cdot (\sigma - p); p \in \sigma\},$$
 with weight function $\omega_\Star(\R_{\geq 0} (\sigma - p)) = \omega_X(\sigma)$. It is easy to see that this is a balanced fan and that its support is equal to
 $$\{v \in \tpn{n}; \textnormal{There is } r_v > 0 \textnormal{ such that } p + \alpha v \in \abs{X} \textnormal{ for all } \alpha \in [0,r_v]\}.$$ 
 \end{itemize}
 
\end{defn}

\begin{figure}[ht]
\centering
\begin{tikzpicture}
 \matrix[column sep = 30pt]{
 \draw (0,0) -- (-1,0);
 \draw (0,0) -- (1,1);
 \draw (0,0) -- (0,-2);
 \draw (-1,-1) -- (1,-1);
 \draw (1,-1) -- (1,-2);
 \draw (1,-1) -- (2,0);
 \fill[black] (0,-1) circle (2pt) node[below left]{\tiny $p$};
 \draw (0,-.5) -- (0,-.5) node[left = 40]{\small $X$};
 &
 \draw (0,0.2) node[right=40] {\small $\rec(X)$} -- (-.6,0.2) node[above]{\tiny 2};
 \draw (0,0.2) -- (.6,.8) node[right]{\tiny 2};
 \draw (0,0.2) -- (0,-.4) node[left]{\tiny 2};
 \draw (0,-2) -- (0,-1);
 \draw (-.5,-1.5) -- (.5,-1.5);
 \draw (0,-1.5) -- (0,-1.5) node[right=40]{\small $\Star_X(p)$};
 \\
 };
\end{tikzpicture}
\caption{Forming the recession fan and the $\Star$ of $X$ at a point. Note that, while $\rec(X)$ is supported on a tropical linear space, it has nontrivial weights, so Theorem \ref{prelim_thm_valuated_matroid} does not apply}.
 
\end{figure}
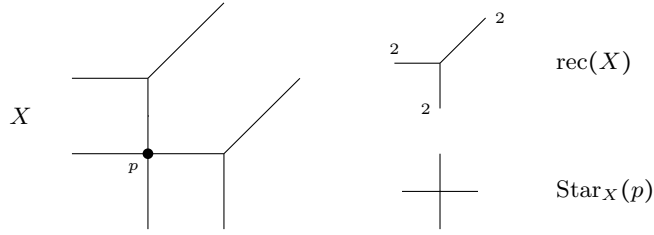

\subsection{Tropical linear spaces}

We will assume that the reader is familiar with the basic notions of matroid theory (see \cite{omatroidtheory} for a comprehensive study of the topic). For a study of tropical linear spaces see for example \cites{frstiefeltropical,risotropical,MaclaganSturmfelsBook,stropicallinear}. To quickly recap the matroid terminology we will mostly use: A \emph{circuit} is a minimal dependent set and a \emph{flat} is a closed set, i.e.\ adding any element increases the rank. Note that we will assume all matroids to be loopfree.

Dress and Wenzel \cite{DressWenzel} generalized the notion of a matroid to that of a \emph{valuated matroid}:

\begin{defn}
 A \emph{valuated matroid} $(M,w)$ is a matroid $M$ on a set $E = \{1,\dots,n\}$ together with a \emph{valuation} $w: \curly{B} \to \R$ on its set of bases $\curly{B}$ fulfilling the \emph{tropical Plücker relations}:
 
 For all $B_1,B_2 \in \curly{B}$ and every $u \in B_1$, there exists a $v \in B_2$, such that both $B_1 - u + v$ and $B_2 -v +u$ are bases and 
  $$(w(B_1) \odot w(B_2)) \oplus (w(B_1 - u  + v) \odot w( B_2 -v + u)) = w(B_1 - u  + v) \odot w( B_2 -v + u).$$
 \end{defn}
 
\begin{remark}
 As in classical matroid theory, there are various equivalent ways of defining a valuated matroid. Another way, discovered by Murota and Tamura \cite{mtcircuitvaluation} is via \emph{valuated circuits}. A circuit valuation is obtained by choosing a vector $v_C \in (\R \cup (-\infty)/\textbf{1})^n$ for each circuit $C$ of $M$ such that the following are fulfilled:
 \begin{itemize}
  \item $C = \{i \in [n]; (v_C)_i \neq -\infty\}$ for all circuits $C$.
  \item Let $C,C'$ be circuits and assume $i \in C \cap C', j \in C\wo C'$. Choose representatives of $v_C, v_{C'}$ such that $(v_C)_i = (v_{C'})_i$. Then there exists a circuit $D$ and a representative of $v_D$ such that $i \notin D$, $(v_D)_j = (v_C)_j$ and $v_D \oplus v_C \oplus v_{C'} = v_C \oplus v_{C'}$.
 \end{itemize}
The paper \cite{mtcircuitvaluation} shows that both axiom sets are cryptomorphic. More precisely, given a valuation $w: \curly{B} \to \R$ on the bases, a circuit valuation can be defined in the following way: Let $C$ be a circuit. Then $C$ is the fundamental circuit with respect to some basis $B$ and an element $i \notin B$. We set
$$(v_C^w)_j := w(B -j+i) - w(B),$$
where $w(B') = -\infty$, if $B'$ is not a basis. We will consider any valuated matroid $(M,w)$ to be equipped with this circuit valuation.
\end{remark}
 
 \begin{defn}
  One can define a polyhedral structure on the set
 $$B(M,w) :=\left\{x \in \tpn{n}; \max_{i \in C}\{x_i + (v_C^w)_i\} \textnormal{ is assumed at least twice $\forall$ circuits }C.\right\}.$$
 and assigning weight 1 to each maximal cell we obtain a tropical variety (a proof can be found in \cite[Theorem 4.4.5]{MaclaganSturmfelsBook}), which we also denote by $B(M,w)$. A \emph{tropical linear space} is a tropical variety of this form.
 \end{defn}
 
 \begin{remark}
 It turns out that we essentially only need to consider trivial valuations (i.e.\ $w \equiv 0$) to prove our theorem, in which case we obtain a polyhedral fan, the \emph{matroid fan} or \emph{Bergman fan} of a matroid $M$:
 $$B(M) := \{x \in \tpn{n}; \max_{i \in C} x_i \textnormal{ is assumed at least twice for all circuits } C \textnormal{ of } M\}.$$
 Note that for any valuation $w$ on a matroid $M$, $B(M) = \rec(B(M,w))$. The key fact that will allow us to reduce the problem to matroidal fans will be the following:
\end{remark}

\begin{theorem}\label{prelim_thm_valuated_matroid}
 Let $X$ be a tropical variety. Then the following are equivalent:
 \begin{itemize}
  \item $X = B(M,w)$ for a valuated matroid $(M,w)$.
  \item $\rec(X) = B(M)$.
 \end{itemize}
\end{theorem}
This follows from the two facts that being a tropical linear space is equivalent to having degree one \cite[Theorem 6.5]{fchowpolytopes} and that a tropical variety and its recession fan have the same degree (see for example the argument in the proof of \cite[Theorem 4.4.5]{MaclaganSturmfelsBook}).

It has been shown that $B(M)$ has several possible representations as a polyhedral fan. We will be working with the structure induced by the flats of $M$ - this is the finest fan structure of $B(M)$ that usually occurs in the literature: 

\begin{defn}\label{matroid_flat_fan}
For a set $F \subseteq [n]$ we write $v_F := -e_F = -\sum_{i \in F} e_i$.

Let $\curly{F}$ be the set of flats of $M$. For any \emph{chain} $\curly{C} = (F_1,\dots,F_d = E)$, where $\emptyset \subsetneq F_1 \subsetneq \dots \subsetneq F_d = E$ and $F_i \in \curly{F}$ for all $i$, we define a polyhedral cone:
$$\cone(\curly{C}) := \{ \sum_{i=1}^{d-1} \lambda_i v_{F_i};\; \lambda_1,\dots,\lambda_{d-1} \geq 0\}.$$
If we go through all chains of flats in $M$, the corresponding cones obviously form a fan and by \cite{akbergman} the support of this fan is $B(M)$.
\end{defn}

\begin{remark}\label{prelim_remark_flats}
 One can also retrieve the matroid from its Bergman fan. It is a well-known fact that its set $\curly{F}$ of flats is $\{F \subseteq [n]; v_F \in B(M)\}$. To see this, assume $v_F \in B(M)$. By \cite{fsmatroid} this is the same as saying that the set of bases $B$ of $M$ such that $\abs{B \cap F}$ is maximal covers all of $E$, as these are the bases of minimal weight with respect to $v_F$. But that implies that $F$ is a flat: If $i \notin F$, there is a basis $B$ containing $i$ and having maximal intersection with $F$, so $\rank(F + i) = \abs{B \cap (F+i)} = \abs{B \cap F} + 1 = \rank(F) +1$.
\end{remark}

The following has been known for long, but we include the proof here for completeness:

\begin{prop}\label{prop_matroid_fan_is_convex}
 Let $(M,w)$ be a valuated matroid of rank $r$ on $n$ elements. Then $B(M,w)$ is tropically convex.
 \begin{proof}
  Let $x,y \in B((M,w))$ and $\lambda, \mu \in \R$. Let
  $$z := \lambda \odot x \oplus \mu \odot y = ( \max\{x_i + \lambda,y_i + \mu\})_{i=1,\dots,n}.$$
 Let $C \subseteq [n]$ be a circuit of $M$. We will denote the corresponding valuation by $v_C$. In that case $\max_{i \in C} \{x_i + (v_C)_i\}, \max_{i \in C} \{y_i + (v_C)_i\}$ are both assumed twice, i.e.\ there exist $i_1 \neq i_2,j_1 \neq j_2$ such that
 \begin{align*}
  x_C &:= \max_{i \in C} \{x_i+(v_C)_i\} = x_{i_1}  + (v_C)_{i_1} = x_{i_2}  + (v_C)_{i_2}\\
  y_C &:= \max_{i \in C} \{y_i+(v_C)_i\} = y_{j_1}  + (v_C)_{j_1} = y_{j_2} + (v_C)_{j_2}
 \end{align*}
We will assume without restriction that $y_C + \mu \geq x_C + \lambda$. Thus, $z_{j_1} + (v_C)_{j_1} = z_{j_2}  + (v_C)_{j_2} = y_C + \mu$. Let $k \in C$ be arbitrary. Then 
\begin{align*}
z_k + (v_C)_k &= \max\{x_k + (v_C)_k + \lambda,y_k + (v_C)_k + \mu\} \\
&\leq \max\{x_C + \lambda, y_C + \mu\} \\
&= y_C + \mu = z_{j_1} + (v_C)_{j_1}.
\end{align*}

In particular, the maximum $\max_{i \in C} \{z_i + (v_C)_i\}$ is assumed twice (at $j_1$ and $j_2$). 

 \end{proof}
\end{prop}

\section{Tropically convex complexes}\label{section_tconv}

\begin{prop}\label{prelim_prop_recession}
 Let $X$ be a polyhedral complex and assume $\abs{X}$ is tropically convex. Then $\abs{\rec(X)}$ is tropically convex as well.
 \begin{proof}
  Assume $v,v' \in \abs{\rec(X)}$. We can reformulate this as the fact that there exist $p,p' \in \abs{X}$ such that $p + \R_{\geq 0}v, p' + \R_{\geq 0}v' \subseteq \abs{X}$. For $\alpha > 0$ we write $q_\alpha := p + \alpha v, q_\alpha' := p' + \alpha v'$ and $\Delta_\alpha :=  q_\alpha' - q_\alpha, \nu := v'-v.$ It is easy to see that we can choose $\alpha$ large enough such that $\type(\Delta_\alpha)$ remains constant and is a \emph{refinement} of $\type(\nu)$, by which we mean that if $\nu_i < \nu_j$, this implies $(\Delta_\alpha)_i < (\Delta_\alpha)_j$. Assume we have fixed such an $\alpha$ and that $\type(\Delta_\alpha) = (I_1,\dots,I_s)$.
  Now as in Remark \ref{prelim_remark_troplinesegment} we see that the tropical convex hull $\tconv\{q_\alpha,q_\alpha'\}$ consists of line segments connecting points
  $$p_j^\alpha := q_\alpha + \sum_{i=2}^j ( \Delta_\alpha(I_{i-1}) - \Delta_\alpha(I_i)) e_{F_{i-1}},$$
  where $F_i = \bigcup_{k \leq i} I_k.$ Now let $\beta > \alpha$. We calculate that for all $j = 1,\dots,s$ we have
  \begin{align*}
   p_j^\beta - p_j^\alpha &= (q_\beta - q_\alpha) + \sum_{i=2}^j \left((\Delta_\beta(I_{i-1}) - \Delta_\beta(I_i)) - (\Delta_\alpha(I_{i-1}) - \Delta_\alpha(I_i)) \right)e_{F_{i-1}}\\
   &= (\beta - \alpha)\underbrace{\left(v + \sum_{i=2}^j \left(\nu(I_{i-1}) - \nu(I_i)\right)e_{F_{i-1}} \right)}_{=: r_j}.
  \end{align*}
Note that $r_1,\dots,r_s$ are exactly the vertices of the line segments forming $\tconv\{v,v'\}$ (some of them may be the same, as $\type(\Delta_\alpha)$ can be strictly finer than $\type(\nu)$). In particular, since $p_j^\beta = p_j^\alpha + (\beta-\alpha)r_j \in \abs{X}$ for any $\beta > \alpha$, we see that $r_j$ lies in $\abs{\rec(X)}$. It is now easy to see that in fact the line segments in between must also lie in $\abs{\rec(X)}$.
 \end{proof}
\end{prop}

\begin{lemma}\label{lemma_reconetoone}
 Let $X$ be a $d$-dimensional polyhedral complex such that $\abs{X}$ is tropically convex. Assume that $\rec(X) = \{\rec(\sigma); \sigma \in X\}$ is a fan. Then for each maximal cone $\rho$ of $\rec(X)$, there is exactly one maximal cell $\sigma$ of $X$ such that $\rho = \rec(\sigma)$.
 \begin{proof}
Assume there are two maximal cells $\sigma, \sigma'$ of $X$ such that $\rho = \rec(\sigma) = \rec(\sigma')$. Pick any two points $p \in \sigma, p' \in \sigma$. For $r \in \rho$ we write $q_r = p + r, q_r' = p' + r$. As taking the tropical convex hull commutes with translations, we see that
$$\tconv\{q_r,q_r'\} = \tconv\{p,p'\} + r.$$
This implies that $\tconv\{p,p'\} + \rho \subseteq \abs{X}.$ But as $p-p' \notin V_\sigma = V_\rho$, there must be a line segment $l \subseteq \tconv\{p,p'\}$, whose slope does not lie in $V_\sigma$. Hence $l + \rho$ is a $(d+1)$-dimensional set contained in $\abs{X}$, which is a contradiction to our assumption.
 \end{proof}
\end{lemma}
\begin{prop}\label{prop_locally_convex}
 Let $X$ be a tropically convex polyhedral complex and $p \in \abs{X}$. Then $\Star_X(p)$ is tropically convex as well.
 \begin{proof}
  As tropical convexity is preserved under translation, we can assume $p = 0$. Let $v,v' \in \Star_X(0)$. It is clear that for any $\alpha > 0, \type( \alpha(v' - v))$ does not change. Hence $\tconv\{\alpha v, \alpha v'\} = \alpha \tconv\{v,v'\}$. As $v,v' \in \Star_X(p)$, the left hand side is contained in $\abs{X}$ for all $\alpha \leq 1$. This implies $\tconv\{v,v'\} \subseteq \abs{\Star_X(p)}$.
 \end{proof}
\end{prop}

\begin{defn}
 Let $X$ be a polyhedral fan in $\tpn{n}$. We write
$$\curly{F}_X := \{ F \subseteq [n]: v_F \in \abs{X}\}.$$
For any chain $\curly{C} = (F_1,\dots,F_d = E)$ in $\curly{F}_X$, we will define $\cone(\curly{C})$ as in Definition \ref{matroid_flat_fan}. Note that each such cone is unimodular and has dimension $d -1$.

We then obtain a polyhedral fan, the \emph{chain fan} of $X$:
$$\textnormal{Ch}_X:= \{ \cone(\curly{C});\; \curly{C} \textnormal{ a chain in } \curly{F}_X\}.$$
In general this fan obviously need not be pure and can be empty. 
\end{defn}

\begin{remark}
A special case is  $X = \tpn{n}$. Then $\textnormal{Ch}_{X} =: \textnormal{Ch}_n$ is a subdivision of $\tpn{n}$ according to partitions, i.e.\ two elements lie in the interior of the same cone, if and only if they have the same partition. In particular, if $x \in \tpn{n}$, then the minimal cone of $\textnormal{Ch}_n$ containing $x$ is $d$-dimensional if and only if $\size(x) = d+1$. Note that $\textnormal{Ch}_n$ can also be defined as the quotient of the normal fan of the permutohedron or as the chains-of-flats subdivision corresponding to the uniform matroid $U_{n,n}$.
\end{remark}

\begin{lemma}\label{lemma_cones_in_fan}
 Let $X$ be a tropically convex polyhedral fan in $\tpn{n}$. Then the following hold:
\begin{enumerate}
 \item Assume $\dim(X) = d$ and $x \in \abs{X}$. Then $\size(x) \leq d+1$. In particular, $X$  is contained in the $d$-dimensional skeleton of $\textnormal{Ch}_n$.
 \item $\curly{F}_X$ is closed under intersections, i.e.\ if $F,F' \in \curly{F}_X$, then $F \cap F' \in  \curly{F}_X$.
 \item If $\curly{C}$ is any chain in $\curly{F}_X$, then $\cone(\curly{C}) \subseteq \abs{X}$.
\end{enumerate}
\begin{proof}\newl
\begin{enumerate}
 \item Fix a representative of $x$ and assume that $s := \size(x) > d+1$. Let $\type(x) = I_1 \cup \dots \cup I_s$ as in Remark \ref{prelim_remark_troplinesegment} and $F_j := \bigcup_{i \leq j} I_i$. The tropical line segment $\tconv\{0,x\}$ consists of segments connecting the points $0 = p_1,\dots,p_s = x$, where
 $$p_j = \sum_{i=2}^j (x(I_{i-1}) - x(I_i)) e_{F_{i-1}}.$$
 We will show inductively that for $k = 2,\dots,s-1$ there are continuous, concave, piecewise affine linear functions $\epsilon_k: \R^{k-1} \to \R$, such that $\epsilon_k$ is strictly positive on $(\R_{>0})^{k-1}$ and such that for $k \geq 1$ we have 
 $$P_k = \left\{\sum_{i=1}^{k} \lambda_i e_{I_i}; \lambda_i \in [0,\epsilon_i(\lambda_1,\dots,\lambda_{i-1})] \textnormal{ for all } i > 1, \lambda_1 \geq 0\right\} \subseteq \abs{X}.$$
 Then $P_{s-1}$ is a polyhedron of dimension $s-1 > d$ in the fan $X$, which is a contradiction to our assumption that $\dim X = d$.
 
 For $k = 1$ we get $P_1 = \{\lambda_1 e_{I_1};\lambda_1 \geq 0\}$. As $p_2 = (x(I_1) - x(I_2)) e_{I_1} \in X$ and $X$ is a fan, this is always contained in $X$. Now assume $k > 1$ and that we have found $\epsilon_2,\dots,\epsilon_{k-1}$. Let $q = \sum_{i=1}^{k-1} \lambda_i e_{I_i} \in P_{k-1}$ and assume $0 < \lambda_i$ for all $i$. In particular, $q(I_i) > 0$ for all $i < k$ and $q(I_i) = 0$ for $i \geq k$. We now consider the tropical line segment $\tconv\{\alpha q,p_{k+1}\}$, where $\alpha > 0$. By induction $q \in \abs{X}$ and since $X$ is a fan, so is $\alpha q$. We conclude that $\tconv\{\alpha q,p_{k+1}\} \subseteq \abs{X}$.
 Let $\Delta = p_{k+1} - \alpha q$. Note that $\Delta$ has constant entries on each $I_j$ (as both $q$ and $p_{k+1}$ do). More precisely, we have
 $$\Delta(I_j) = \begin{cases}
    0, &\textnormal{if } j > k\\
    x(I_k) - x(I_{k+1}), &\textnormal{if } j = k\\
    x(I_j) - x(I_{k+1}) - \alpha \lambda_j, &\textnormal{if } j < k.
   \end{cases}
$$
 
 Now, if we pick $\alpha$ sufficiently large, $\Delta$ is maximal on entries in $I_k$: For $j \neq k$ we have
$$
  \Delta(I_k) - \Delta(I_j) = 
  \begin{cases}
      x(I_k) - x(I_{k+1}) > 0, &\textnormal{if } j > k\\
      \underbrace{x(I_k) - x(I_j)}_{ <0} + \alpha \underbrace{\lambda_j}_{>0}, &\textnormal{if } j < k
     \end{cases}
 $$
and the latter is always greater than zero if we pick 
$$\alpha > \max_{j < k} \left\{ \frac{x(I_1) - x(I_{k+1})}{\lambda_j}\right\} =: m_{k,\lambda}.$$
In this case the first line segment of $\tconv\{\alpha q,p_{k+1}\}$, going from $\alpha q$ to $p_{k+1}$, has slope $e_{I_k}$. Its length is $d_k := x(I_k) - x(I_{k+1})$: Our choice of $\alpha$ implies that $\Delta(I_j) < 0$ for all $j < k$.

In particular, $\{\alpha q + \lambda e_{I_k}; \lambda \in [0,d_k]\} \subseteq \abs{X}$. But that implies that $q + \lambda e_{I_k} \in \abs{X}$, whenever 
$$0 \leq \lambda \leq d_k \cdot \frac{1}{m_{k,\lambda}} = d_k \cdot \min_{j<k}\left\{ \lambda_j \cdot \frac{1}{x(I_1) - x(I_{k+1})}\right\} =: \epsilon_k(\lambda_1,\dots,\lambda_{k-1}).$$
As $\epsilon_k$ clearly fulfills all desired properties, the claim follows.
 \item Follows directly from the fact that $v_F \oplus v_{F'} = v_{F \cap F'}$.
 \item Let $\curly{C} = (F_1,\dots,F_l = E)$ be a chain in $\curly{F}_X$. An element of $\cone(\curly{C})$ is of the form $w = \sum_{i=1}^{l-1} \lambda_i v_{F_i}$ for some $\lambda_i \geq 0$. We define $\mu_j = \sum_{i=j+1}^{l-1} \lambda_j$ for $j = 0,\dots,l-1$ and claim that
 $$w = \bigoplus_{j=1}^{l-1} (- \mu_j) \odot (\mu_0 \cdot v_{F_j}),$$
 which by assumption is an element of $\abs{X}$. Indeed, let $k \in \{1,\dots,n\}$ and note that $\mu_0 \geq \mu_1 \geq \dots \geq \mu_{l-1} = 0$. Then
 \begin{align*}
  \left( \bigoplus_{j=1}^{l-1} (- \mu_j) \odot (\mu_0 \cdot v_{F_j})\right)_k &= \max_{j=1,\dots,l-1} \left( \mu_0 \cdot (v_{F_j})_k - \mu_j\right)\\
  &= \begin{cases}
      - \mu_0,&\textnormal{if } k \in F_1\\
      - \mu_{j(k)},\textnormal{where } j(k) :=  \max\{j: k \notin F_j\},&\textnormal{otherwise}
     \end{cases}
%
\\
  &= - \sum_{i: k \in F_i}\lambda_i\\
  &= w_k.
 \end{align*}

\end{enumerate}
\end{proof}
\end{lemma}

\section{Tropical convexity and valuated matroids}\label{section_proof}

By Proposition \ref{prop_matroid_fan_is_convex} we only need to prove that any tropically convex tropical variety is supported on a tropical linear space. In fact, it will suffice to reduce to the case of fans:

\begin{prop}\label{proof_prop_fan}
 Let $X$ be a tropical variety whose support is a tropically convex fan. Then $\abs{X} = B(M)$ for a matroid $M$.
\end{prop}

Our main theorem now follows from this:

\begin{proof}(of Theorem \ref{main_theorem})
 The \enquote{if} direction follows from Proposition \ref{prop_matroid_fan_is_convex}. For the \enquote{only if} direction, let $X$ be a tropical variety with tropically convex support. By Propositions \ref{prop_locally_convex} and \ref{proof_prop_fan}, $\Star_X(p)$ is supported on a matroidal fan $B(M(p))$ at each $p \in \abs{X}$. Since any matroidal fan is irreducible \cite[Lemma 2.4]{frdiagonalintersection}, we must have $\Star_X(p) = k_p \cdot B(M(p))$ for some $k_p \in \N$. 
 As $\abs{X}$ is tropically convex, it is also path-connected. This implies that $k_p = k_{p'} =:k$ for all $p,p' \in \abs{X}$. We conclude that $X = k \cdot Y$ for some $k \in \N$ and a tropical variety $Y$ with constant weight 1.
 
 By Propositions \ref{prelim_prop_recession} and \ref{proof_prop_fan}, $\rec(X)$ is also supported on a matroidal fan and hence of the form $l \cdot B(M)$ for some $l \in \N$ and some matroid $M$. Using Lemma \ref{lemma_reconetoone}, we see that $\rec(Y) = B(M)$ (and in fact: $k = l$). By Theorem \ref{prelim_thm_valuated_matroid} we must have $Y = B(M,w)$ for some valuation $w$ on $M$, so $\abs{X} = \abs{Y} = \abs{B(M,w)}$, as claimed.
 
\end{proof}

The general idea for proving Proposition \ref{proof_prop_fan} is to revert the procedure described in Remark \ref{prelim_remark_flats}: We define $M$ via its flats, which is the set $\curly{F}_X$ of all sets whose incidence vectors lie in $\abs{X}$. We show that $X$ is supported on the fan of chains of $\curly{F}_X$. Then it only remains to show that $\curly{F}_X$ actually fulfills the axioms required for a set of flats. Naturally, the balancing condition plays a crucial role in the proofs of both statements.

\begin{prop}\label{prop_chain_cone_support}
 Let $X$ be a tropical fan and assume $\abs{X}$ is tropically convex. Then 
 $$\abs{X} = \abs{\textnormal{Ch}_X}.$$
\begin{proof}
 Let $d := \dim X$ and let $\curly{X}$ be a polyhedral structure of $X$. By Lemma \ref{lemma_cones_in_fan},(1), $\abs{X}$ is contained in the $d$-dimensional skeleton of $\textnormal{Ch}_n$. By intersecting $X$ with $\textnormal{Ch}_n$, we can now assume that each cone of $\curly{X}$ is contained in some $d$-dimensional $\cone(\curly{C})$, where $\curly{C}$ is a chain of arbitrary subsets of $[n]$.
 
The balancing condition of $X$ now dictates that if $\rho = \cone(\curly{C}) \in \textnormal{Ch}_n$ contains a maximal cone of $\curly{X}$ in its interior, all of $\rho$ must be in $\abs{X}$:
 Otherwise, it would contain a codimension one face $\tau$ of $\curly{X}$ in its interior, such that there is one maximal cone $\sigma > \tau$ with $\sigma \subseteq \rho$ but no other maximal cone $\sigma' > \tau$ is contained in $\rho$. Balancing implies that at least one other maximal cone $\sigma'$ is adjacent to $\tau$. By our previous argument, $\sigma$ and $\sigma'$ both lie in $d$-dimensional cones of $\textnormal{Ch}_n$. But these cones intersect only in their boundary. 

This proves $\abs{X} \subseteq \abs{\textnormal{Ch}_X}$ and the converse follows from Lemma \ref{lemma_cones_in_fan},(3).
\end{proof}
\end{prop}

\begin{proof}(of Proposition \ref{proof_prop_fan})

Proposition \ref{prop_chain_cone_support} tells us that we can equip $X$ with the polyhedral structure of $\textnormal{Ch}_X$: Balancing implies that all cones contained in a cone of $\textnormal{Ch}_X$ must have the same weight. Hence it is sufficient to show that $\curly{F}_X$ defines indeed a set of flats of a matroid. More precisely, we have to show the following:
\begin{enumerate}
 \item $E \in \curly{F}_X$.
 \item If $F,F' \in \curly{F}_X$, then $F \cap F' \in \curly{F}_X$.
 \item Let $F \in \curly{F}_X$ and assume $F_1,\dots,F_k$ are the minimal elements of $\curly{F}_X \wo \{F\}$ that contain $F$. Then $E \wo F$ is the disjoint union of $F_1 \wo F, \dots, F_k \wo F$.
\end{enumerate}
The first statement is trivial and the second follows from Lemma \ref{lemma_cones_in_fan}, (2).

To prove the third axiom, let $F \in \curly{F}_X$ and denote by $F_1,\dots,F_k$ the minimal elements of $\curly{F}_X$ containing it. Then $(F_i \wo F) \cap (F_j \wo F) = \emptyset$ by minimality and the second axiom. Hence we only have to prove that $\bigcup_{i=1}^k F_i = E$. By Proposition \ref{prop_chain_cone_support} and the fact that $X$ is pure of some dimension $d$, every maximal chain in $\curly{F}_X$ has the same length $d+1$. Let $G \in \curly{F}_X$ and $\curly{C} = (F_1,\dots,F_{d+1} = E)$ a maximal chain in $\curly{F}_X$. We define the \emph{rank} of $G$ to be $\rank(G) := i$, if $F_i = G$. This is independent of the actual chain: Otherwise we could combine two chains with $G$ occurring at different positions to form a chain of length greater than $d+1$. 

We will now prove the last axiom by induction on $c(F) := d + 1 - \rank(F)$. If $c(F) = 1$, then $k = 1$ and $F_1 = E$, so the statement is true. Now let $c(F) > 1$ and $j \in E \wo F$. By induction, there exists an $F' \in \curly{F}_X$ of rank $\rank(F) +2$ and with $F \subseteq F'$, such that $j \in F'$. We can now pick a chain
$$\curly{D} = (G_1,\dots,G_{c(F)} = F,G_{c(F)+2} = F',\dots,G_{d+1} = E),$$
where $\rank(G_i) = i$ for all $i$ (such a chain exists, as $X$ is pure). Then $\cone(\curly{D})$ is a codimension one cone. Using the fact that all chain cones are unimodular, the balancing equation at $\cone(\curly{D})$ reads:
$$v := \sum_{G \in \curly{G}(\curly{D})} w_G v_G \in V_{\cone(\curly{D})},$$
where $\curly{G}(\curly{D}) := \{G \in \curly{F}_X; F \subsetneq G \subsetneq F'\} \subseteq \{F_1,\dots,F_k\}$ and $w_G \in \Z \wo \{0\}$ denotes the weight of the corresponding maximal cone. 

As $v \in V_{\cone(\curly{D})}$, all entries $\{v_i, i \in F' \wo F\}$ agree (this notion is obviously well-defined in $\tpn{n}$). If we pick the obvious representative $- \sum_{i \in G} e_i \in \R^n$ for each $v_G$ and use the fact that for any $F_s,F_t \in \curly{G}(\curly{D})$ we have $(F_s \wo F) \cap (F_t \wo F) = \emptyset$, then for $i \in F' \wo F$ we have
$$v_i = \begin{cases}
         - w_G, &\textnormal{ if there is a } G \in \curly{G}(D) \textnormal{ with } i \in G\\
         0, &\textnormal{ otherwise.}
        \end{cases}$$

As $X$ is pure, $\curly{G}(\curly{D})$ is not empty. But this implies that there must be an $F_s \in \curly{G}(\curly{D})$ with $j \in F_s$.
%
\end{proof}

\section{Local-to-global tropical convexity}\label{section_local}

In classical convexity theory, there are various local-to-global principles. In this section, we prove Theorem \ref{conv_theorem}, a tropical analogue of a result proven by Tietze and Nakajima \cites{tkonvexheit,nkonvexekurven}. It states that any closed connected subset of $\R^n$, which is locally convex, is already convex. The main strategy of the proof follows a standard argument for classical convexity - though there is some extra work involved due to the fact that $\tpn{n}$ is not uniquely geodesic with its canonical metric (see Remark \ref{remark_geodesic}).

\begin{defn}
 Let $x \in \tpn{n}$. We define the \emph{tropical norm} of $x$ to be 
 $$\tnorm{x} := \max\{x_i\} - \min\{x_i\}.$$
 We also fix the following notations:
 \begin{align*}
  B_r^\trop(x) &:= \{y \in \tpn{n}; \tnorm{y-x} \leq r\},\\
  \partial B_r^\trop(x) &:= \{y \in \tpn{n}; \tnorm{y-x} = r\},\\ 
  \Imin (x) &:= \{i \in [n]; x_i \textnormal{ minimal}\},\\
  \Imax (x) &:= \{i \in [n]; x_i \textnormal{ maximal}\}.\\
 \end{align*}
For a compact set $S$ and a point $x$, we will also write
$$\tnorm{x-S} := \min\{\tnorm{x-s}; s \in S\}.$$
 
\end{defn}

The following all have easy and elementary proofs:
\begin{lemma}\newl
\begin{enumerate}
 \item $\tnorm{\cdot}$ is twice the quotient norm of the maximum norm on $\R^n$. In particular, it defines a norm on the $\R$-vector space $\tpn{n}$. 
 \item Let $r > 0$ and $P_{r,n}$ be the cube in $\R^n$ with vertices  $re_F, \emptyset \subseteq F \subseteq [n]$. Then $B_r^\trop(0)$ in $\tpn{n}$ is the image of $P_{r,n}$ under the quotient map. 
 \item $B_r^\trop(x)$ is a \emph{polytrope}, i.e.\ convex and tropically convex.
 \item If $x \in \tpn{n}$ and $\tconv\{0,x\}$ consists of actual line segments connecting points $0 = p_1,\dots,p_s = x$, then
 $$\tnorm{x} = \sum_{i=1}^{s-1} \tnorm{p_{i+1} - p_i}.$$
\end{enumerate} 
\end{lemma}

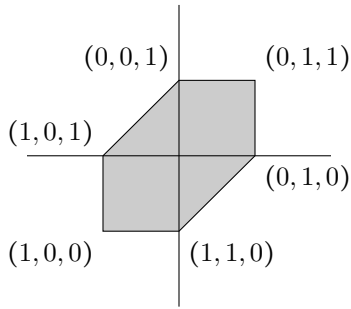
\begin{figure}[ht]
 \begin{tikzpicture}
  
  \filldraw[fill=light-gray] (-1,0) node[above left] {$(1,0,1)$}-- (0,1) node[above left]{$(0,0,1)$} -- (1,1) node[above right]{$(0,1,1)$} -- (1,0) node[below right]{$(0,1,0)$}-- (0,-1) node[below right] {$(1,1,0)$}-- (-1,-1) node[below left]{$(1,0,0)$}-- (-1,0);
  \draw (-2,0) -- (2,0);
  \draw (0,-2) -- (0,2);
 \end{tikzpicture}
\caption{The unit sphere $B_1^\trop(0)$ in $\tpn{3}$. The picture is drawn in two dimensions by setting the first coordinate to 0.}
\end{figure}

\begin{remark}\label{remark_geodesic}
 The tropical norm was already introduced in \cite{dstropicalconvexity} to study tree metrics. Joswig shows that $\tpn{n}$ is a \emph{geodesic} space \cite{jtropicalcombinatorics}: The tropical line segment between two points is a geodesic with respect to the metric induced by $\tnorm{\cdot}$. However, it is not uniquely geodesic. There are generally various paths from $x$ to $y$ whose length is $\tnorm{x-y}$ (see also Figure \ref{figure_intersect}).
 \begin{figure}[ht]
  \centering
  \begin{tikzpicture}
    \matrix{
   \filldraw[fill=light-gray] (-1.5,0) -- (0,1.5)-- (1.5,1.5) -- (1.5,0) -- (0,-1.5) -- (-1.5,-1.5) -- (-1.5,0);
   \filldraw[fill=light-gray] (1.5,1) -- (3,2.5) -- (4.5,2.5) -- (4.5,1) -- (3,-.5) -- (1.5,-.5) -- (1.5,1);
   \draw[line width=3pt] (1.5,0)-- (1.5,1);
   \draw (0,0) -- (2,0) -- (3,1);
   
   \fill[black] (0,0) node[left]{$x$} circle (2pt);
   \fill[black] (3,1) node[right]{$y$} circle (2pt);&
    \\
    };
  \end{tikzpicture}
  \caption{The set of points $z$ with $\tnorm{x-z} = \tnorm{z-y} = \tnorm{x-y}/2$ is a polytrope.}\label{figure_intersect}
 \end{figure}
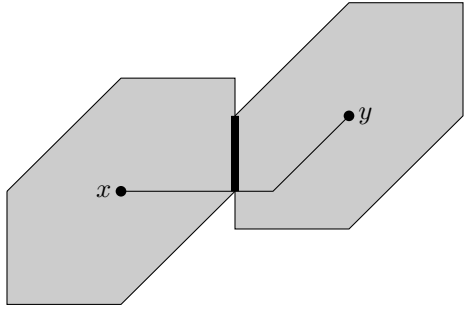

\end{remark}

\begin{lemma}\label{lemma_imin}
 Let $x,y,z \in \tpn{n}$. Then for any point $p \in \tconv\{x,y\}$, we have
 $$\Imax(p-z) \subseteq \Imax(x-z) \cup \Imax(y-z).$$
 \begin{proof}
  Let $j \in \Imax(p-z)$ and assume $j \notin \Imax(x-z)$. We know that we can write $y = x + \sum_{i=1}^s \alpha_i e_{F_i}$, where $\alpha_i > 0$ and $F_i \subsetneq F_{i+1}$ for all $i$ and each summand corresponds to a vertex on the tropical line segment $\tconv\{x,y\}$. Hence $p = x + \sum_{i=1}^{k-1} \alpha_i e_{F_i} + \beta e_{F_k}$ for some $k \leq s$ and $\beta \leq \alpha_k$. Since $p_j - z_j = (x_j - z_j) + \sum_{i=1}^{k-1} \alpha_i (e_{F_i})_j + \beta (e_{F_k})_j$ is maximal and $x_j - z_j$ is not maximal, we must have that $j \in F_m$ for some $m \leq k$. In particular, $j \in F_l$ for all $l \geq k$, so $y_j - z_j$ is still maximal.
 \end{proof}
\end{lemma}

\begin{defn}
 Let $X \subseteq \tpn{n}$ and $x,y \in X$.
 \begin{itemize}
  \item We call $X$ \emph{locally tropically convex}, if for every $x \in X$ there exists an $\epsilon >0$, such that $B_\epsilon^\trop(x) \cap X$ is tropically convex.
  \item A \emph{tropical path} in $X$ from $x$ to $y$ is an injective continuous map $\gamma:[0,1] \to X$, whose image is a concatenation of tropical line segments leading from $x$ to $y$. The \emph{length} $l(\gamma)$ of $\gamma$ is the length with respect to $\tnorm{\cdot}$, i.e. if $\gamma$ consists of tropical line segments connecting $x = x_0, \dots, x_k = y$, then 
 $$l(\gamma) = \sum_{i=1}^k \tnorm{x_i - x_{i-1}}.$$
 \item We define the \emph{distance} of $x$ and $y$ in $X$ to be
 $$d_X(x,y) := \inf\{l(\gamma); \gamma \textnormal{ a tropical path from $x$ to $y$.}\}.$$
 \end{itemize}

\end{defn}

\begin{lemma}\label{lemma_threepoints}
 Let $X \subseteq \tpn{n}$ be locally tropically convex and $x,y \in X$. Assume there is a point $z \in X$ such that the following hold:
 \begin{itemize}
  \item $d_X(x,z) = d_X(z,y) = d_X(x,y)/2$.
  \item $\tnorm{z- \tconv\{x,y\}}$ is minimal among all points fulfilling the first property.
  \item $\tconv\{x,z\}, \tconv\{y,z\} \subseteq X$.  
 \end{itemize}
Then $z \in \tconv\{x,y\}$, so $\tconv\{x,y\} \subseteq X$.
\begin{proof} Assume $z \notin \tconv\{x,y\}$. We define $F := \Imax(x-z), F' := \Imax(y-z)$. Then $e_F, e_{F'}$ are the outgoing slopes of the tropical line segments from $z$ to $x$ and $y$, respectively (see also Figure \ref{figure_closepoint} for an illustration). 
 
 Now choose $\epsilon >0$ small and let $$z' := (z + \epsilon e_F) \oplus (z + \epsilon e_{F'}) = z + \epsilon e_{F \cup F'}.$$ By local tropical convexity, this lies in $X$ for sufficiently small $\epsilon$.
 
 First of all, we see that $z'$ still fulfills the first property: Note that the concatenation of $\tconv\{x,z + \epsilon e_F\}$ and $\tconv\{z + \epsilon e_F,z'\}$ forms a tropical path in $X$ from $x$ to $z'$. Again assuming $\epsilon$ to be sufficiently small and using that $x \neq z$, we get 
 \begin{align*}
  d_X(x,z') &\leq \tnorm{x - (z + \epsilon e_F)} + \tnorm{z' - (z + \epsilon e_F)}\\
  &\leq (\tnorm{x-z} -\epsilon) + \epsilon\\
  &= \tnorm{x-z} = d_X(x,z).
 \end{align*}
Similarly, $d_X(y,z') \leq d_X(y,z)$. But as $z$ was already a midpoint, this implies equality. 

We now claim that $\tnorm{z' - \tconv\{x,y\}} < \tnorm{z - \tconv\{x,y\}} =:l$, which is a contradiction to our assumption. To see this, let $$M := \{p \in \tconv\{x,y\}; \tnorm{z-p} = l\}.$$ As $M = \tconv\{x,y\} \cap B_l^\trop(z)$, it is tropically convex. We know by Lemma \ref{lemma_imin}, that for any point $p$ in $M$, we have $\Imin(z-p) \subseteq F \cup F'$. Also note that by assumption $z \notin M$. We will now prove that $\Imax(z-p) \cap (F \cup F') = \emptyset$.

Assume $M = \{p\}$ is only a point. Then moving from $p$ along the tropical line segment $\tconv\{x,y\}$ strictly increases the distance to $z$. Let $G := \Imax(x-p), G' := \Imax(y-p)$. Then for small $\epsilon'$ we must have
$$\tnorm{z - (p+ \epsilon'e_G)} = \tnorm{(z-p) - \epsilon' e_G}> \tnorm{z-p}.$$
So if $\emptyset \neq \Imax(z-p) \cap F = \Imax(z-p) \cap \Imax(x-z)$, then we must have $G = \Imax(z-p) \cap F$. But then
$$\tnorm{z - (p+ \epsilon'e_G)} \leq \tnorm{z-p},$$
which is a contradiction. The same argument works for $F'$ and $G'$, so we see that $\Imax(z-p) \cap (F \cup F') = \emptyset$.

If $M$ is a tropical line segment, we can choose $p$ such that $\tconv\{x,y\}$ is locally at $p$ a line with slope $v_G$, $G:= \Imax(x-p)$ and 
$$\tnorm{z - (p \pm \epsilon' e_G)} = \tnorm{z-p}.$$
But this is only possible if either $\Imin(z-p) \cup \Imax(z-p) \subseteq G $ or $(\Imin(z-p) \cup \Imax(z-p)) \cap G = \emptyset$. So if $\emptyset \neq \Imax(z-p) \cap F$, we must again have $G = \Imax(z-p) \cap F$, so neither of the above two possibilities would hold. Again, the same argument works for $F'$.

In either case, we see that $\Imax(z-p) \cap (F \cup F') = \emptyset$ and $\Imin(z-p) \subseteq (F \cup F')$. But then 
$$\tnorm{z'-p} = \tnorm{(z-p)+\epsilon e_{F \cup F'}} < \tnorm{z-p}.$$
This contradicts our assumption. Hence we must have $z \in \tconv\{x,y\}$. 
\end{proof}
\end{lemma}
\begin{figure}[ht]
 \centering
 \begin{tikzpicture}[scale=2]
  \fill[fill=light-gray] (0.2,0) -- (1,0) -- (1.8,.8) -- (1.8,1.6) -- (1,1.6) -- (0.2,.8) -- (0.2,0);
  \draw[gray] (1.8,.8) -- (1.8,1.6) -- (1,1.6) -- (0.2,.8) -- (0.2,0);
  \draw (1.8,1.6) node [below left] {\small $B_l^\trop(z)$};
  \draw (0,0)-- (1,0) -- (2,1);
  \draw (1.8,0.8) -- (1,0.8) -- (0.2,0);
  \fill[black] (0,0) circle (1pt) node[left] {$x$};
  \fill[black] (2,1) circle (1pt) node[right] {$y$};
  \fill[black] (1,.8) circle (1pt) node[left] {$z$};
  \draw[->, line width=2pt] (1,.8) -- (0.5,0.3) node[above = 5pt] {$e_F$};
  \draw[->, line width=2pt] (1,.8) -- (1.5,.8) node[above] {$e_{F'}$};
  \draw[dashed] (.5,.3) -- (1,.3) -- (1.5,.8);
  \fill[black] (1,.3) circle (1pt) node[above]{$z'$};
  \fill[black] (1.5,.5) circle (1pt) node[below right]{$p$};
  \draw[->, line width=2pt] (1.5,.5) -- (1.25,0.25) node[below right] {$e_G$};
  \draw[->, line width=2pt] (1.5,.5) -- (1.75,0.75) node[below right] {$e_{G'}$};

 \end{tikzpicture}
  \caption{Constructing a point closer to the tropical line segment using local tropical convexity.}\label{figure_closepoint}
\end{figure}
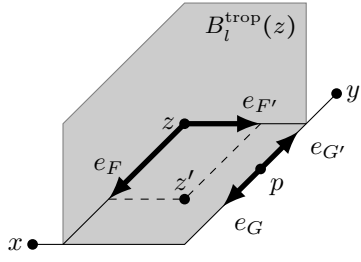

\newpage

 \begin{proof}(of Theorem \ref{conv_theorem})
  Let $x,y \in X$ and set $r := d_X(x,y)$. 
  
  Since $X$ is closed and locally tropically convex, there exists a midpoint, i.e. a point $z \in X$ such that
  $$d_X(x,z) = d_X(y,z) = \frac{1}{2}r.$$

%
  Note that the set of midpoints of $x$ and $y$ is a compact set. It is obviously closed and must be a subset of $B_{r/2}^\trop(x) \cup B_{r/2}^\trop(y)$. Hence we can choose $z$ to have minimal distance to $\tconv\{x,y\}$.
  
  In this manner, we recursively construct points $z_{i,n} \in X$ with $1 \leq n, 0 \leq i \leq 2^n$, such that
  \begin{itemize}
   \item $z_{0,n} = x, z_{n,n} = y$ and $z_{i,n} = z_{2i,n+1}$.
   \item $z_{2i+1,n+1}$ is a midpoint of $z_{i,n}$ and $z_{i+1,n}$ and it has minimal distance to $\tconv\{z_{i,n}, z_{i+1,n}\}$.
  \end{itemize}

  In particular, we have $d_X(z_{i,n},z_{i+1,n}) = r/2^n$. Now we have
  \begin{align*}
   \tnorm{z_{i,n} - x} &\leq d_X(z_{i,n},x) \leq (i/2^m) r \leq r
  \end{align*}
  for all $i$ and $n$, so $z_{i,n} \in B_r^\trop(x) \cap X =: B$, which is a compact set. Hence we can choose a global $\delta > 0$ such that for all $x \in B$, the set $B_\delta^\trop(x) \cap X$ is tropically convex.
  
  By choosing $n$ large enough, we can now assume that $r/2^{n+1} < \delta$. Then for each $i$, $B_\delta^\trop(z_{2i+1,n+1})$ contains both $z_{i,n}$ and $z_{i+1,n}$, so their tropical convex hull is contained in $X$. Applying Lemma \ref{lemma_threepoints} inductively, we see that $\tconv\{x,y\} \subseteq X$.

 \end{proof}

\begin{corollary}
 Let $X$ be a connected tropical variety in $\tpn{n}$, which is locally a multiple of a matroidal fan, i.e.\ $\Star_X(p) = k_p \cdot B(M(p))$ for each $p \in X$, some $k_p \in \Z$ and some matroid $M(p)$. Then $X$ is supported on a tropical linear space.
\end{corollary}

\newpage

\printbibliography

\end{document}

%% file: definition.tex
\usepackage[style=american]{csquotes}
\usepackage{yfonts}
\usepackage{amsmath,amsthm,amscd,amssymb}
\usepackage{MnSymbol}
\usepackage[all]{xy}
\usepackage{smartref}
\usepackage{algorithmic}
\usepackage{tikz}
\usepackage[hyperindex,breaklinks]{hyperref}
\usepackage{algorithm, caption}
\usepackage{faktor}

\theoremstyle{definition}
\newtheorem{defn}{Definition}[section]
\newtheorem{ex}[defn]{Example}

\newtheorem{remark}[defn]{Remark}

\newtheorem*{acknowledgement}{Acknowledgement}
\newtheorem*{convention}{Convention}

\theoremstyle{plain}
\newtheorem{theorem}[defn]{Theorem}
\newtheorem{corollary}[defn]{Corollary}
\newtheorem{lemma}[defn]{Lemma}
\newtheorem{prop}[defn]{Proposition}

\newcommand{\abs}[1]{\left\lvert #1 \right\rvert} 
\newcommand{\gnrt}[1]{\left\langle #1 \right\rangle} 

\newcommand{\Z}{\mathbb{Z}}

\newcommand{\N}{\mathbb{N}}
\newcommand{\R}{\mathbb{R}}

\newcommand{\Star}{\textnormal{Star}}

\newcommand{\curly}[1]{\mathcal{#1}}

\newcommand{\wo}{\setminus} 

\definecolor{comment}{rgb}{.2,.2,.2}

\newcommand{\trop}{\textnormal{trop}}

\newcommand{\newl}{\indent\par}

\numberwithin{equation}{section}
\definecolor{DarkGreen}{rgb}{0,0.5,0}
\definecolor{DarkRed}{rgb}{0.8,0,0}